\documentclass[preprint,12pt]{elsarticle}

\usepackage{amssymb}
\usepackage{amsthm}
\usepackage{tikz}
\usepackage{amsmath}
\usepackage{amscd}
\usepackage{framed}
\newtheorem{Theorem}{Theorem}[section]
\newtheorem{Proposition}[Theorem]{Proposition}
\newtheorem{Lemma}[Theorem]{Lemma}

\newcommand{\R}{{\mathbb R}}
\newcommand{\C}{{\mathbb C}}
\newcommand{\N}{{\mathbb N}}

\newcommand{\V}{{\mathbb V}}

\journal{Journal of Differential Equations}

\begin{document}

\begin{frontmatter}



\title{Density of Schr\"odinger Weyl-Titchmarsh m functions on Herglotz functions}


\author{Injo Hur}

\address{Mathematics Department, University of Oklahoma, Norman, OK, USA, 73019\\
 and Mathematics Department, Sogang University, Seoul, Republic of Korea, 04107}

\begin{abstract}
We show that the Herglotz functions that arise as Weyl-Titchmarsh $m$ functions of one-dimensional Schr\"odinger operators are dense in the space of all Herglotz functions with respect to uniform convergence on compact subsets of the upper half plane. This result is obtained as an application of de Branges theory of canonical systems. 
\end{abstract}


\end{frontmatter}


\section{Introduction}\label{secintro}





We discuss the very natural question of whether any arbitrary Herglotz function can be approximated by Weyl-Titchmarsh $m$ functions \cite{Titch,Weyl} of Schr\"odinger operators  $S=-d^2/dx^2+V(x)$. These $m$ functions, called \textit{Schr\"odinger $m$ functions}, have several descriptions which are essentially obtained by the inverse spectral theories of these operators \cite{Borg,GL, GS2,Hor,Lev,Mar1,Mar2,RSim,RemAf,RemdB,SimIST}. Informally, these results reveal that a Herglotz function is a Schr\"odinger $m$ function precisely if it has the right large asymptotics.

It is not clear, however, how to express precise conditions along these lines. The descriptions in these theories have some Gelfand-Levitan type conditions which are usually stated in terms of a Fourier-Laplace type transform of the spectral measure. All this means that the problem above seems to lead to some difficult issues about how to enforce certain asymptotic behavior on approximating functions, with additionally not knowing what exactly we are trying to enforce.

Our goal in this paper is to present a transparent way to avoid these difficulties and, therefore, answer the problem positively. The crucial machinery for the new path is de Branges theory of canonical systems \cite{Ach,deB1,deB2,deB3,deB4,deB,Dym,HSW,KL,RemdB,Sakh,Win,Win2}, which enables us to rephrase the problem as one about approximations of canonical systems. 

More precisely, there are two well known facts in the theory; the first is the one-to-one correspondence between Herglotz functions and (trace-normed) canonical systems in  \cite{deB,Win}, and the second is the fact that trace-normed canonical systems are always in a limit point case at $\infty$ in \cite{Ach,deB2}. The latter is, especially, used to cook up a topology on canonical systems which interacts very well with Herglotz functions. 

The last piece of information for the route is the characterization of some special canonical systems. They are called \textit{Schr\"odinger canonical systems}, since they are some disguised Schr\"odinger equations such that each of them and its corresponding Schr\"odinger equation share their $m$ functions. This means that it is unnecessary to check if the $m$ functions associated with Schr\"odinger canonical systems are Schr\"odinger $m$ functions. They already are. This characterization, moreover, shows how to construct them. Besides de Branges theory, it is another key to deal with canonical systems, not with their $m$ functions. 

Based on the three ingredients above let us present our method for the problem in terms of canonical systems. 
See also the box below. For given a Herglotz function, by de Branges theory, there exists a unique trace-normed canonical system whose $m$ function is the given Herglotz function. It is then approximated by tailor-made Schr\"odinger canonical systems in the sense of the topology on canonical systems working well with Herglotz functions. This implies that their Schr\"odinger $m$ functions  converge to the given Herglotz function with respect to uniform convergence on compact subsets of the upper half plane.

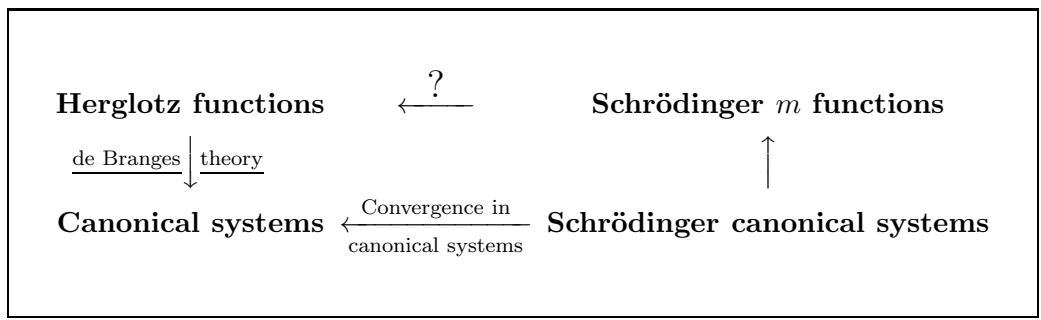
\begin{framed}
\begin{center}
\small{
$\begin{CD}
\textbf{Herglotz functions}    @<\textrm{\large{?}}<<  \textbf{Schr\"odinger $m$ functions}\\
@V\textrm{\underline{de Branges}}V\textrm{\underline{theory}}V @AAA\\
\textbf{Canonical systems}   @<\textrm{Convergence in}<\textrm{canonical systems}<   \textbf{Schr\"odinger canonical systems}
\end{CD}$
}
\end{center}
\end{framed}

We would like to enlighten two more things. The main result can become stronger since it will be shown that all Weyl-Titchmarsh $m$ functions corresponding to Schr\"odinger operators with both \textit{smooth} potentials and  some \textit{fixed} boundary condition at 0 are dense in the space of all Herglotz functions. This will be discussed more clearly in Sections \ref{secScs} and, essentially,  \ref{secpf}. The other is that the procedure above is so general that it has a lot of potential to give ideas about unsolved questions of Schr\"odinger operators in the viewpoint of (inverse) spectral theory.\\

This paper is organized as follows.  The following section provides basic materials about Schr\"odinger operators, canonical systems and their $m$ functions. In Section \ref{secDensity} the main result is stated with several comments. We then, in Section \ref{secScs}, characterize all Schr\"odinger canonical systems, that is, all canonical systems which can be written as Schr\"odinger equations. As the last preparation, a topology on canonical systems is made up in Section \ref{sectop}. However, not to lose our main theme, several continuous properties are verified in Appendix A. The stronger result is finally proven in Section \ref{secpf} that all Schr\"odinger $m$ functions with a fixed boundary condition at 0 and smooth potentials are dense in the space of the Herglotz functions. \\

\section{Preliminaries}\label{secpre}
\subsection{Schr\"odinger operators and their $m$ functions} 
Let us start with one-dimensional Schr\"odinger operators 
\begin{equation}\label{Schop}
S=-\frac{d^2}{dx^2}+V(x)
\end{equation}
on $L^2(0,b)$, where $0<b<\infty$ or $b=\infty$, and $V$ are real-valued locally integrable functions, called potentials.   Schr\"odinger eigenvalue equations associated with (\ref{Schop}) is 
\begin{equation}
\label{Se}
-y''(x,z)+V(x)y(x,z)=zy(x,z), \quad x\in(0,b)
\end{equation}
where $z$ is a spectral parameter. It is then well known that each operator (\ref{Schop}), or equivalently each equation (\ref{Se}), with boundary condition(s) at 0 and possibly at $b$ has  a unique Weyl-Titchmarsh $m$ function and vice versa \cite{Borg,Mar1,Titch,Weyl}. 

More precisely,  put a boundary condition at 0, 
\begin{equation}
\label{bcat0}
y(0) \cos\alpha-y'(0) \sin\alpha=0
\end{equation} 
where $\alpha \in [0,\pi)$. For $0<b<\infty$, we place another boundary condition at $b$, 
\begin{equation}
\label{bcatb}
y(b) \cos\beta+y'(b) \sin\beta=0
\end{equation}
with another real number $\beta$ in $[0,\pi)$. Note that $\beta$ is used as a parameter for (\ref{bcatb}). When $b=\infty$, Weyl theory says that, if (\ref{Schop}) is in a limit point case at $\infty$, no more boundary condition except (\ref{bcat0}) is needed. However, if (\ref{Schop}) is in a limit circle case at $\infty$, that is, every solution of (\ref{Se}) is in $L^2(0,\infty)$, then it is necessary to have a limit type boundary condition at $\infty$ as follows:  Put $f(x,z):=u_{\alpha}(x,z)+m(z)v_{\alpha}(x,z)$, where $u_{\alpha}$ and $v_{\alpha}$ are the solutions to (\ref{Se}) satisfying the initial conditions, $u_{\alpha}(0,z)=v_{\alpha}'(0,z)=\cos\alpha$ and $-u_{\alpha}'(0,z)=v_{\alpha}(0,z)=\sin\alpha$. Then $m(z)$ is on the limit circle if and only if 
\begin{equation}
\label{bcatblcc}
\lim_{N\to\infty} W_N(\bar{f}, f)=0
\end{equation}
where $W_N$ is the Wronskian at N, that is, $W_N(f,g)=f(N)g'(N)-f'(N)g(N)$ and $\bar{f}$ is the complex conjugate of $f$. Similar to the case when $0<b<\infty$, $\beta$ is made use of as a parameter for these boundary conditions at $\infty$. See \cite{CodLev,Weid} for more details.

Then (\ref{Schop}) with (\ref{bcat0}) and possibly either (\ref{bcatb}) or (\ref{bcatblcc}) has a unique $m$ function $m^S_{\alpha, \beta}$  and it can be expressed by 
\begin{equation}
\label{mfnforSe}
m^S_{0, \beta}(z)=\frac{\tilde{y}'(0,z)}{\tilde{y}(0,z)}\quad \textrm{or}
\quad
m^S_{\alpha, \beta}(z)=\begin{pmatrix}  \cos\alpha & \sin\alpha \\ -\sin\alpha & \cos\alpha \end{pmatrix} \cdot
m^S_{0, \beta}(z)
\end{equation}
where $\tilde{y}$ is a solution to (\ref{Se}) which is square-integrable near $\infty$ when (\ref{Schop}) is in a limit point case at $b=\infty$, or which is satisfying either (\ref{bcatb}) when $0<b<\infty$ or (\ref{bcatblcc}) when (\ref{Se}) is in a limit circle case at $b=\infty$. Here $\cdot$ means the action of a 2$\times$2 matrix as a linear fractional transformation (which will be reviewed soon).  
For convenience $m_{\alpha,\beta}^S$ are called Schr\"odinger $m$ functions, as talked. They are \textit{Herglotz functions}, that is, they map the upper half plane $\C^+$ holomorphically to itself. See e.g. \cite{LS} for all these properties of $m_{\alpha,\beta}^S$.

Before going further, let us recall the action of linear fractional transformations, based on \cite{Remweyl}. A \textit{linear fractional transformation} is a map of the form 
\begin{equation*}
z\mapsto \frac{az+b}{cz+d}
\end{equation*}
with $a,b,c,d\in\C$, $ad-bc\neq 1$. This can be expressed very easily via matrix notation by 
\[
A\cdot z=\frac{az+b}{cz+d}, \quad A=\begin{pmatrix} a&b \\ c&d \end{pmatrix}.
\]
This notation has a natural interpretation: Identify $z\in\C\subset \mathbb{CP}^1$ with its homogeneous coordinates $z=[z:1]$ and apply the matrix $A$ to the vector 
$(\begin{smallmatrix} z\\1 \end{smallmatrix})$ 
whose components are these homogeneous coordinates. The image vector 
$A(\begin{smallmatrix} z\\1 \end{smallmatrix})$ 
then reveals what the homogeneous coordinates of the image of $z$ under the linear fractional transformation are. 

These remarks also show that the mapping 
\[
A\mapsto \textrm{linear fractional transformation}
\]
is a group homomorphism between the general linear group $GL(2,\C)$ and the non-constant linear fractional transformations, which implies that $\cdot$ can be thought of as the action of linear fractional transformations. The homomorphism property will be used in Section \ref{secScs}. Let us also mention that in (\ref{mfnforSe}) the special orthogonal group $SO(2,\R)$ is only considered among $GL(2,\C)$.\\

Even though Schr\"odinger $m$ functions are Herglotz functions, the converse is not true. To verify this let us see that, because of the Herglotz representation, not all Herglotz functions can have the asymptotic behavior which Schr\"odinger $m$ functions should do. Indeed,  Everitt \cite{Eve} showed that, when $z\in\C^+$ is large enough,  $m_{\alpha,\beta}^S$ satisfy the asymptotic behavior  
\begin{equation}
\label{asymm1}
m_{0,\beta}^S(z)=i\sqrt{z}+o(1) 
\end{equation}
for $\alpha=0$, or
\begin{equation}
\label{asymm2}
m_{\alpha,\beta}^S(z)= \frac{\cos\alpha}{\sin\alpha}+\frac1{\sin^2 \alpha}\frac{i}{\sqrt{z}}
+O \big( |z|^{-1} \big)
\end{equation}
for $\alpha\in (0,\pi)$. See also \cite{Atk,Har} for more developed versions of the asymptotic behavior of $m_{\alpha,\beta}^S$. 
Given a Herglotz function $F$, it can be expressed by  
\begin{equation}
\label{Herglotz}
F(z)=A+\int_{\R_{\infty}} \frac{1+tz}{t-z} d\rho(t)
\end{equation}
where $A$ is a real number and $d\rho$ is a finite positive Borel measure on $\R_{\infty}$, the one-point compactification of the set of all real numbers $\R$. (See e.g.  (2.1) in \cite{Remac}.) Then (\ref{Herglotz}) indicates that any Herglotz function with a measure $d\rho$ having a positive point mass at $\infty$ cannot satisfy (\ref{asymm1}) nor (\ref{asymm2}), and therefore it is not a Schr\"odinger $m$ function. However, for $d\rho$ to be a measure associated with  (\ref{Schop}) (or so called spectral measure), a more issue is on the asymptotic behavior of $d\rho$ near $\infty$. See two sections 17 and 19 of \cite{RemdB} for details.\\

\subsection{Canonical systems and de Branges theory} 
To see a general connection between Herglotz functions and differential equations let us consider a half-line canonical system,  
\begin{equation}
\label{cs}
Ju'(x,z)=zH(x)u(x,z), \quad x\in(0,\infty)
\end{equation}
 where $H$ is a positive semidefinite $2\times2$ matrix whose entries are real-valued, locally integrable functions and $J=\big( \begin{smallmatrix}  0 & -1 \\ 1 &0 \end{smallmatrix}$\big). A canonical system (\ref{cs}) is called \textit{trace-normed} if $\textrm{Tr }H(x)=1$ for almost all $x$ in $(0,\infty)$. For (\ref{cs}) we always place a boundary condition at 0,
\begin{equation}
\label{bcat0forcs}
u_1(0,z)=0
\end{equation}
where $u_1$ is the first component of $u=\big( \begin{smallmatrix}  u_1 \\ u_2\end{smallmatrix}$\big). Similar to 
(\ref{mfnforSe}),  its $m$ function, $m_H$, can be expressed by 
\begin{equation}
\label{mfnforcs}
m_H(z)=\frac{\tilde{u}_2(0)}{\tilde{u}_1(0)}
\end{equation}
where $\tilde{u}=\big( \begin{smallmatrix} \tilde{u}_1 \\ \tilde{u}_2 \end {smallmatrix} \big) $ is a solution to (\ref{cs}) satisfying   
\begin{equation}
\label{H-int}
\int_0^{\infty} \tilde{u}^*(x)H(x)\tilde{u}(x)dx<\infty.
\end{equation}
Here $^*$ means the Hermitian adjoint. Such a solution satisfying (\ref{H-int}) is called \textit{$H$-integrable}. See \cite{Win,Win2} for all these properties of (\ref{cs}). 

Recall that there were three cases when defining Schr\"odinger $m$ functions and in each case we needed a special solution to formulate the corresponding $m$ function. For (\ref{cs}) an $H$-integrable solution, however, is only needed, since (\ref{cs}) is half-line and a half-line trace-normed canonical system is always in a limit point case at $\infty$. In other words, there is only one $H$-integrable solution up to a multiplicative constant. See the original argument by \cite{deB2} or an alternative proof in \cite{Ach} for more details.\\

De Branges \cite{deB} and Winkler \cite{Win} then showed that, for a given Herglotz function, there exists a unique half-line trace-normed canonical system with (\ref{bcat0forcs}), such that its $m$ function $m_H$ is the given Herglotz function. This one-to-one correspondence is essential later in order to cope with canonical systems rather than Herglotz functions or their $m$ functions.\\

\section{Main result}\label{secDensity}
In this paper, we show the density of Schr\"odinger $m$ functions on all Herglotz functions or, equivalently, all  $m$ functions $m_H$ to (\ref{cs}) in the sense of their natural topology as analytic functions on $\C^+$ by the following theorem.

\begin{Theorem}\label{Density}
The space of Schr\"odinger $m$ functions with some fixed boundary condition at 0 is dense in the space of all Herglotz functions. 
\end{Theorem}

The above theorem is stronger than what we just said, since, as the statement itself says, all Schr\"odinger $m$ functions \textit{with any fixed boundary condition at 0} are dense in all Herglotz functions. The result can, moreover, be stronger with such Schr\"odinger $m$ functions corresponding to only \textit{smooth} potentials, which will be clear in Section \ref{secpf} after proving. As its application, Schr\"odinger $m$ functions with the Dirichlet boundary condition at 0, $m^S_{0,\beta}$, corresponding to smooth potentials are dense in all Herglotz functions. 
 
Due to Theorem \ref{Density} it cannot be expected that Schr\"odinger operators with some fixed boundary condition at 0 converge to some Schr\"odinger operator in the sense that this convergence is equivalent to the uniform convergence of their $m$ functions on compact subsets of $\C^+$. This is one of reasons why only subclasses of Schr\"odinger operators are considered in many applications to make them compact.\\

Remark. It seems very difficult to show Theorem \ref{Density} through $m$ functions directly, even though the measures corresponding to (\ref{Schop}), called Schr\"odinger spectral measures, are dense in the space of all measures in the Herglotz representation with respect to weak-$\ast$ convergence. In (\ref{Herglotz}) we can see that the uniform convergence of Herglotz functions on compact subsets of $\C^+$ is equivalent to both the weak-$*$ convergence of the measures $d\rho$ and the pointwise convergence of the constants $A$. In particular, the weak-$*$ convergence of measures (without the convergence of constants) is not sufficient for the convergence of Herglotz functions. 

It turns out that any finite positive Borel measure on $\R_{\infty}$ can be approximated by Schr\"odinger spectral measures in the weak-$*$ sense. Indeed, for any finite positive Borel measure $d\rho$ on $\R_{\infty}$, construct a sequence of measures $d\rho_n$ by 
\begin{equation*}
d\rho_n(t)=\chi_{(-n,n)}(t) d\rho(t)+\chi_{\R \setminus (-n,n)}(t) d\rho_{free}(t)+\rho\{\infty\}\delta_n(t)
\end{equation*}
where $\delta_n$ is a Dirac measure at $n$ and $d\rho_{free}$ is the spectral measure for (\ref{Schop}) with $V\equiv 0$. In other words, this is a sequence of truncated measures having the tail of $d\rho_{free}$ and putting the point mass at $n$ with the weight $\rho\{\infty\}$, which implies that $d\rho_n$ are Schr\"odinger spectral measures. Then $d\rho_n\to d\rho$ in the weak-$*$ sense, as $n\to\infty$. The weak-$*$ convergence of spectral measures $d\rho_n$, however, does not imply the convergence of $m$ functions associated with $d\rho_n$ in (\ref{Herglotz}). This is because any Schr\"odinger spectral measures determine their $m$ functions; the error term in (\ref{asymm1}) or (\ref{asymm2}) is at least $o(1)$, which means that spectral measures  decide the corresponding constants in (\ref{Herglotz}). Therefore it is unclear if these constants converge to some constant, and even worse we cannot see if they will converge to the constant corresponding to a given Herglotz function.\\

\section{Schr\"odinger canonical systems}\label{secScs}
It is well known that Schr\"odinger equations can be expressed by some canonical systems (which will be shown later) but the converse is not true. This is the reason why canonical systems are thought of as generalizations of Schr\"odinger equations. For us it is, however, necessary to learn which canonical systems admit Schr\"odinger $m$ functions as their $m$ functions. In this section, let us figure out all the conditions for such canonical systems, called Schr\"odinger canonical systems as before. 
 
\begin{Proposition}
\label{Cor}
A Schr\"odinger equation (\ref{Se}) with boundary conditions (\ref{bcat0}) and, if necessary, either (\ref{bcatb}) or (\ref{bcatblcc}) can be expressed as the following canonical system such that both have the same Weyl-Titchmarsh $m$ functions: 
\begin{equation}\label{Sc}
J \frac{d}{dt}u(t,z)=z P_{\varphi}(t) u(t,z) , \quad t\in(0,\infty) 
\end{equation} 
with 
\begin{equation}
\label{HforS}
P_{\varphi}(t):=\begin{pmatrix} \cos^2\varphi(t) & \cos \varphi(t) \sin \varphi(t) \\ \cos \varphi(t) \sin \varphi(t) & \sin^2\varphi(t) \end{pmatrix}. 
\end{equation}
Here a new variable $t$ is defined by 
\begin{equation}
\label{t}
t(x)=\int_0^x \big( u_{0}^2(s)+v_{0}^2(s) \big) \textrm{ }ds
\end{equation}
where $u_{0}$ and $v_{0}$ are the solutions to the given Schr\"odinger equation for $z=0$ with $u_0(0)=v'_0(0)=\cos\alpha$ and $-u'_0(0)=v_0(0)=\sin\alpha$. 
Put $t_b:= \lim_{x\uparrow b}t(x)$ in $(0,\infty]$. Then $\varphi$ is a strictly increasing function on $(0,t_b)$, which has a locally integrable third derivative on $(0,t_b)$, satisfying three initial conditions
\begin{equation}\label{initialcondition}
 \varphi(0)=\alpha,\quad \frac{d\varphi}{dt}(0)=1, \textit{ and }\quad \frac{d^2\varphi}{dt^2}(0)=0.
\end{equation}
If $t_b<\infty$, then $\varphi(t)=\tilde{\beta}$ on $(t_b, \infty)$ for some real number $\tilde{\beta}\in [0,\pi)$. 

Conversely, any canonical system (\ref{Sc}) with all the properties of $\varphi$ above can be written as (\ref{Se}) with some locally integrable potential $V$ such that they have the same $m$ function.\\
\end{Proposition}

In short, the proposition reveals that Schr\"odinger equations are exactly the canonical systems with  projection matrices $P_{\varphi}$ as their $H$ in (\ref{cs}), such that $\varphi$ are strictly increasing functions on $(0,t_b)$ having the third derivatives which have the same regularity with potentials $V$, and they behave as linear functions with slope 1 near 0. If $t_b<\infty$, $\varphi$ are constant on $(t_b,\infty)$. Moreover, their function values at 0, $\varphi(0)$, are the same as $\alpha$ in (\ref{bcat0}) up to multiples of $\pi$. 

Note that, to have the condition $\varphi(0)=\alpha$ precisely, we assume that $\varphi(0)$ are in $[0,\pi)$. This is fine because all entries in $P_{\varphi}$ are periodic with the period $\pi$,  which means that $\varphi$ can be shifted by multiples of $\pi$ at our disposal. 

 The reason why obtaining the one-to-one correspondence between (\ref{Se}) and (\ref{Sc}) is that their $m$ functions are the same. Without this restriction, it is possible to connect (\ref{Se}) to infinitely many different trace-normed canonical systems which, of course, take different $m$ functions from each other.

 It is well known how to convert (\ref{Se}) to (\ref{cs}) (see e.g.  \cite{Achtw} or \cite{RemdB}) and the converse of Proposition \ref{Cor} is a kind of reformulation of proposition 8.1 in \cite{RemdB}  in terms of trace-normed canonical systems. It is also realized that a very similar form to (\ref{Sc}) was discussed in \cite{L&W} and \cite{W&W} to deal with semibounded canonical systems. However, (\ref{Sc}) with (\ref{HforS}) is a specific form which exactly fits to Schr\"odinger canonical systems. It is also efficacious, since $\varphi$ can be thought of as a spectral data. For example, since the third derivative of $\varphi$ and $V$ have the same regularity and (\ref{Sc}) can be well defined with \textit{measurable} $\varphi$, singular potentials may be treated by dropping the regularity of $\varphi$. 

The reason to have projection matrices $P_{\varphi}$ is the asymptotic behavior of solutions to (\ref{Se}).  Indeed, de Branges, Krein and Langer \cite{deB2,KL} showed that solutions to (\ref{cs}) belong to Cartwright class of the exponential type $h$ with
\begin{equation*}
h=\int_0^x \sqrt{\textrm{det }H(t)}dt
\end{equation*}
for fixed $x$. An entire function $F$ belongs to \textit {Cartwright class of the exponential type $h$} if 
\begin{equation*}
h:=\limsup_{|z|\to\infty}\frac{\textrm{ln }|F(z)|}{|z|} \textrm{ is finite}, 
\end{equation*}
and 
\begin{equation*}
\int_{-\infty}^{\infty}\frac{ | \textrm{ln } |F(x)| |}{1+x^2}<\infty. 
\end{equation*}
P\"oschel and Trubowitz \cite{P&T} then showed that solutions to (\ref{Se}) are of order 1/2 as entire functions with respect to $z$ for fixed $x$. See also (4.3) in \cite{RemdB}. In particular, they are of exponential type 0 and so are the solutions to (\ref{cs}) associated with them (by (\ref{connection}) below), which implies that  $\textrm{det }H=0$. Since $H$ are symmetric, the two conditions, $\textrm{Tr }H(x)=1$ and $\textrm{det }H(x)=0$ for almost all $x$, indicate that $H$ should be projection matrices $P_{\varphi}$ after some change of variables.\\ 

 Let us now verify Proposition \ref{Cor}. 
\begin{proof}[Proof of Proposition \ref{Cor}]
Let $y$ be a solution to a given Schr\"odinger equation (\ref{Se}). Define $u=u(x,z)=(u_1(x,z), u_2(x,z))^t$ by 
\begin{equation}
\label{connection}
 \begin{pmatrix} u_1 \\ u_2 \end{pmatrix}:=
\begin{pmatrix} u_0(x) & v_0(x) \\ u_0'(x) & v_0'(x) \end{pmatrix}^{-1}\begin{pmatrix} y \\ y' \end{pmatrix}.
\end{equation}
Note that this is well defined, since the determinant of the $2\times 2$ matrix in (\ref{connection}) is the Wronskian of $u_0$ and $v_0$ at $x$, $W_x(u_0,v_0)$, which is 1 for all $x$; in particular, this matrix is invertible. Then $u$ solves (\ref{cs}) with 
\begin{equation*}
H_0(x):= \begin{pmatrix} u_0^2(x) & u_0(x)v_0(x) \\ u_0(x)v_0(x) & v_0^2(x) \end{pmatrix}. 
\end{equation*}
This is shown by direct computation, which is left to readers. 

The matrix $H_0(x)$ may not be trace-normed, and it should be changed to a trace-normed matrix in order to apply de Branges theory. 
For this we do a change of variable as follows; define $R$ and $\varphi$ through
\begin{equation}
\label{defofR}
 u_0(x)+iv_0(x):=R(x) (\cos\varphi(x)+i\sin\varphi(x))
\end{equation}
 and a new variable $t$ by (\ref{t}). Then $u(t,z)$ solves (\ref{Sc}) with (\ref{HforS}), but only on $(0,t_b)$.\\

Let us investigate all the conditions of $\varphi$ in the proposition. Direct computation with (\ref{defofR}) shows the key relation  
\begin{equation}
\label{Wron}
(1=) \textrm{ } W(u_0,v_0)|_{x}=R^2(x)\varphi'(x), \quad x\in(0,b), 
\end{equation} 
which tells us that $\varphi$ is strictly increasing on $(0,b)$ with respect to $x$. Due to three equalities $u_0''=Vu_0$, $v_0''=Vv_0$ and $u^2_0+v^2_0=R^2$, the functions $V$, $u_0''$, $v_0''$, $R''$ and $\varphi'''$ are locally integrable. See also (\ref{formulaforV}) below. These are because of two following facts; the derivatives of solutions to (\ref{Se}) are absolutely continuous when $V$ is locally integrable, and two linearly independent solutions cannot be zero at the same time. The initial values of $u_0$ and $v_0$ can also be transformed to the conditions $\varphi(0)=\alpha$, $R(0)=1$, and $R'(0)=0$, or equivalently, $\varphi(0)=\alpha$, $\varphi'(0)=1$, and $\varphi''(0)=0$ by direct computation. Note that $\varphi$ is normalized by the condition $\varphi(0)\in[0,\pi)$, as mentioned.

So far all the conditions of $\varphi$ have been verified with respect to $x$. Now that $dt/dx=u^2_0+v^2_0=R^2$ and $R$ cannot be zero, these conditions can be converted to the ones with the new variable $t$. In other words, it can be shown that, as a function of $t$,  $\varphi$ is a strictly increasing function on $(0,t_b)$ satisfying (\ref{initialcondition}) whose third derivative is locally integrable. The details are left to readers.\\ 

Note that $t_b<\infty$ precisely if $b<\infty$ or (\ref{Se}) is in a limit circle case at $b=\infty$, since all solutions to (\ref{Se}) are in $L^2(0,b)$ in these two cases. 
When $t_b<\infty$, it is not difficult to see that the boundary condition (\ref{bcatb}) or (\ref{bcatblcc}) can be converted to a similar one 
\begin{equation}\label{bcforcs}
u_1(t_b) \cos(\tilde{\beta})+u_2(t_b) \sin(\tilde{\beta})=0
\end{equation}
for $u$  with another number $\tilde{\beta}\in [0,\pi)$. To obtain a \textit{half-line} canonical system we change (\ref{bcforcs}) to a singular interval $(t_b, \infty)$ of type $\tilde{\beta}$, in other words, $H=P_{\tilde{\beta}}$ on this interval, where $P_{\tilde{\beta}}$ is (\ref{HforS}) with $\tilde{\beta}$ instead of $\varphi(t)$. 

Observe that, for all $t\in[t_b,\infty)$, 
\begin{equation}\label{trivialext}
u(t)=u(t_b)
\end{equation}
and 
\begin{equation}\label{te}
u^*(t)P_{\tilde{\beta}}\textrm{ }u(t)=0. 
\end{equation}
Indeed, if $u$ satisfies (\ref{bcforcs}), $P_{\tilde{\beta}}\textrm{ }u(t_b)$ is the $2\times 1$ zero matrix. With this, since $I_{2}-zLJP_{\tilde{\beta}}$ is the transfer matrix on $(t_b,\infty)$ for a nonnegative number L (see e.g. section 10 in  \cite{RemdB}), where $I_2$ is the $2\times2$ identity matrix, we have that   
\begin{equation*}
u(t_b+L)=(I_2-zLJP_{\tilde{\beta}})u(t_b)=u(t_b),
\end{equation*}
which implies (\ref{trivialext}) and (\ref{te}). All this means that, if a solution $u$ satisfies (\ref{bcforcs}), then it can be trivially extended to $[t_b,\infty)$ such that 
\begin{equation}\label{extofu}
\int_{t_b}^{\infty} u^*(t)P_{\tilde{\beta}}u(t) dt=0.
\end{equation}\\

So far we have constructed (\ref{Sc}) with (\ref{HforS}) as their $H$. It remains to show that both (\ref{Se}) and (\ref{Sc}) have the same $m$ function. To see this let us compare their solutions which were used to define their $m$ functions. Recall $\tilde{y}$ in (\ref{mfnforSe}), that is, $\tilde{y}$ is a solution to (\ref{Se}) which is square-integrable near $b=\infty$ or satisfying either (\ref{bcatb}) or (\ref{bcatblcc}), and let $\tilde{u}$ be the solution to (\ref{Sc}) corresponding to $\tilde{y}$ through (\ref{connection}). Then either $\tilde{u}$ or its trivial extension (again denoted by $\tilde{u}$)  through (\ref{trivialext}) is $H$-integrable with $H=P_{\varphi}$. More clearly, since $\tilde{y} \in L^2[0,b)$, that is, 
\begin{equation*}
\int_0^{b} \tilde{y}(x)^{*}\tilde{y}(x) dx<\infty, 
\end{equation*}
by (\ref{connection}), the above $L^2$-condition is equivalent to the $H_0$-integrability only on $(0,b)$, i.e., 
\begin{equation*}
\int_0^{b} \tilde{u}(x)^*H_0(x) \tilde{u}(x) dx<\infty. 
\end{equation*}
 The change of variable to the new variable $t$ then gives us the condition with respect to $t$ 
\begin{equation*}
\int_0^{t_b} \tilde{u}(t)^*P_{\varphi}(t) \tilde{u}(t) dt<\infty. 
\end{equation*}
If $t_b=\infty$, that is, (\ref{Se}) is in a limit point case at $\infty$, then $\tilde{u}$ is trivially $P_{\varphi}$-integrable. When $t_b<\infty$, the trivial extension of $\tilde{u}$ is $P_{\varphi}$-integrable due to (\ref{extofu}). In other words, since $\tilde{u}$ satisfies (\ref{bcforcs}), its trivial extension $\tilde{u}$ by (\ref{trivialext}) to $(0,\infty)$  is $P_{\varphi}$-integrable. 

Let us now compare their $m$ functions. Recall that $m^S_{\alpha, \beta}$ are the $m$ functions of (\ref{Se}) with (\ref{bcat0}) and, if necessary, either (\ref{bcatb}) or (\ref{bcatblcc}), and $m_H$ are the $m$ functions of (\ref{Sc}) with $H=P_{\varphi}$.  By (\ref{mfnforSe}), (\ref{mfnforcs}) and (\ref{connection}), we then see that  
\begin{eqnarray*}
m_H(z)=m_{P_{\varphi}}(z) &=& \frac{\tilde{u}_2(0)}{\tilde{u}_1(0)}\\
&=& \begin{pmatrix} 0 & 1 \\ 1 & 0 \end{pmatrix} \cdot \frac{\tilde{u}_1(0)}{\tilde{u}_2(0)}\\
&=& \begin{pmatrix} 0 & 1 \\ 1 & 0 \end{pmatrix} 
\begin{pmatrix} u_0(0) & v_0(0) \\ u_0'(0) & v_0'(0) \end{pmatrix}^{-1}
\begin{pmatrix} 0 & 1 \\ 1 & 0 \end{pmatrix} 
\cdot
\frac{\tilde{y}'(0,z)}{\tilde{y}(0,z)}\\
&=&
\begin{pmatrix}  \cos\alpha & \sin\alpha \\ -\sin\alpha & \cos\alpha \end{pmatrix} \cdot
\frac{\tilde{y}'(0,z)}{\tilde{y}(0,z)}\\
&=& m^S_{\alpha, \beta}(z)
\end{eqnarray*} 
where $\cdot$ is the action of a 2$\times$2 matrix as a linear fractional transformation, which was reviewed in Section \ref{secpre}.
Therefore (\ref{Sc}) has been constructed from (\ref{Se}), as desired.\\

For the converse, let us go through the previous process, but in reverse. Assume that (\ref{Sc}) with (\ref{HforS}) be given such that $\varphi$ has all the properties in Proposition \ref{Cor}. If $\varphi$ is constant on a unbounded interval $(c,\infty)$ for some number $c$,  it is possible to put a suitable boundary condition at $c$, which is similar to (\ref{bcforcs}). When $\varphi$ is strictly increasing on $(0,\infty)$, put $c=\infty$, which implies that the corresponding Schr\"odinger operator (\ref{Schop}) will be in a limit point case at $\infty$. 

Since $\frac{d\varphi}{dt}>0$ on $(0,c)$, let us recognize a variable $x$  by 
\begin{equation*}
x(t)=\int_0^t \left[ \frac{d}{dt}\varphi(s) \right]^{1/2} ds   
\end{equation*}
on $(0,x_c)$, where $x_c:=\lim_{t\uparrow c} x(t)$. 
By putting $R(x):= \left[ \frac{d}{dt} \varphi(t(x)) \right] ^{-1/4}$ we can see that $t'(x)=R^2(x)$ and $R^2(x)\varphi'(x)=1$ (here $'$ means $\frac{d}{dx}$). As expected, let's define $u_0$ and $v_0$ by $u_0(x)=R(x) \cos\varphi(x)$ and $v_{0}(x)= R(x)\sin \varphi(x)$. Direct computation then shows that $u$ satisfies (\ref{cs}) with $H_0(x)$ and $R^2(x)\varphi'(x)=W_x(u_0,v_0)$ for all $x$. The nonzero constant Wronskian condition, moreover, allows us to define $y$ through  (\ref{connection}). Then $y$ satisfies (\ref{Se}) with the potential $V$ 
\begin{equation}
\label{formulaforV}
V=\frac7{16}\frac{(\varphi''(x))^2}{(\varphi'(x))^3}-\frac14\frac{\varphi'''(x)}{(\varphi'(x))^2}-\varphi'(x)\quad
\Big(=\frac{R''}{R}-\frac1{R^4} \Big)
\end{equation}
by direct computation, which is left to readers.
 
Similar to the previous argument, it is possible to compare their solutions and then to show that their $m$ functions are the same. Proposition \ref{Cor} now has been proven, as desired.\\ 
\end{proof}

\section{Topology on canonical systems}\label{sectop}
The last preparation is to construct a topology on the set of trace-normed canonical systems (\ref{cs}), which interacts well with the convergence on their $m$ functions.
Let $\V_+$ be the set of the matrices $H$ on trace-normed canonical systems, that is, 
\begin{equation*}
\V_+=\{ H \textrm{ in } (\ref{cs}): \textrm{ Tr }H(x)= 1 \textrm{ for almost all } x\in (0,\infty) \}.
\end{equation*}
Recall that $H$ is a positive semidefinite $2\times2$ matrix whose entries are real-valued, locally integrable functions. Let us say that \textit{$H_n$ converges to $H$ weak-$*$},  if   
\begin{equation}
\label{weak*conv}
\int_0^{\infty}f^*H_nf\to\int_0^{\infty}f^*Hf
\end{equation}
for all continuous functions $f=(f_1,f_2)^t$ with compact support of $[0,\infty)$, as $n\to\infty$. Observe that, for such a given function $f$, two convergences 
\begin{equation*}
\int_0^{\infty}f^*H_nf\to\int_0^{\infty}f^*Hf \quad \textrm{and }  \int_0^{\infty}H_nf\to\int_0^{\infty}Hf
\end{equation*}
are equivalent.\\

By the similar argument in section 2 of \cite{Remcont}, it is briefly shown that $\V_+$ is a compact metric space. First proceed to define a metric on $\V_+$: pick a countable dense (with respect to $|| \cdot ||_{\infty}$) subset $\{ f_n: n\in\N \}$, continuous functions of compact support, and put 
\begin{equation*}
d_n(H_1, H_2):= \Big| \int_{(0,\infty)} f_n^*(x) (H_1-H_2)(x) f_n(x)\textrm{ } dx \textrm{ } \Big|. 
\end{equation*}
Then define a metric $d$ as
\begin{equation*}
d(H_1, H_2):=\sum_{n=1}^{\infty} 2^{-n} \frac{d_n(H_1, H_2)}{1+d_n(H_1, H_2)}. 
\end{equation*}
Clearly, $d(H_n,H)\to 0$ if and only if $H_n$ converges to $H$ weak-$*$, as $n\to\infty$. To show that $(\V_+,d)$ is compact, let us  choose a sequence $H_n$ in $\V_+$. By the Banach-Alaoglu Theorem (on finite intervals $[0,L]$ for some positive numbers $L$) and a diagonal process (for the half line $[0,\infty)$) it is possible to find a subsequence $H_{n_j}$ with the property that the measures $H_{n_j}(t)dt$ converge to some matrix-valued measure $d\mu$ in the weak-$*$ sense. The proof can now be completed by noting that the trace-normed condition, $\textrm{ Tr }H(x)= 1$, is preserved in the limiting process, which implies that the limit measure $d\mu$ is absolutely continuous with respect to the Lebesgue measure and it can be expressed by $H(t)dt$ for some $H$ in $\V_+$.\\

The topology on $\V_+$ above works fine with $m$ functions by the following proposition. 
\begin{Proposition}\label{CovofH}
The map from $\V_+$ to $\overline{\mathbb H}$, defined by $H\mapsto m_H$, is a homeomorphism, where $\mathbb H$ is the set of all (genuine) Herglotz functions and $\overline{\mathbb H}=\mathbb H \cup\R\cup \{ \infty \}$.
\end{Proposition}
It is well known that $\overline{\mathbb H}$ is compact with the uniform convergence on compact subsets of $\C^+$ which is a natural topology for Herglotz functions as analytic functions on $\C^+$. As discussed in Section \ref{secpre}, this map is a bijection by de Branges \cite{deB} and Winkler \cite{Win}. Since $\V_+$ is compact, it suffices to show that this map is continuous. Roughly speaking, this map should be continuous because of Weyl theory and the fact that $\textrm{Tr }H=1$ implies that (\ref{cs}) is in a limit point case at $\infty$. In other words, $H$ on $(0,L)$ for a sufficiently large number $L>0$ almost determines its $m$ function $m_H$.

 Not to miss a major picture, the proof of Proposition \ref{CovofH} is postponed to Appendix A, since it is quite long and the comment above is reasonable.\\

Let us also mention that an equivalent convergence for $H$ to (\ref{weak*conv}) was discussed in \cite{deB2}. More precisely, de Branges showed that, if $n\to\infty$, the convergence $m_{H_n}(z)\to m_{H}(z)$ holds locally uniformly on $C^+$ if and only if 
\begin{equation}\label{lucforH}
\int_0^x H_n(t) dt \to \int_0^x H(t)dt \quad \textrm{locally uniformly for } x\in [0,\infty)
\end{equation}
(see also proposition 3.2 in \cite{L&W}). In de Branges' version, $H_n$ do not need to be trace-normed, but the local uniform convergence is required as payment. Note that, due to the trace-normed condition, the weak-$\ast$ convergence in (\ref{weak*conv}) implies the local uniform convergence in (\ref{lucforH}), which reveals that two convergences are equivalent. In this paper, (\ref{weak*conv}) will be only discussed when proving Proposition \ref{CovofH}.\\ 

\section{Proof of Theorem \ref{Density}}\label{secpf}
In this section let us prove Theorem \ref{Density}. Based on the discussion (or the box) in the introduction, almost all the pieces have already been in our hands from the previous sections. The remaining is to construct Schr\"odinger canonical systems which converge to the trace-normed canonical system whose $m$ function is a given Herglotz function in the sense of the topology on canonical systems in the previous section.

To do this, observe that, since any symmetric matrix can be expressed by the sum of projections by the spectral theorem, it is at least locally and averagely  that $H$ is the sum of projection matrices. In other words, $H$ is $P_{\varphi}$ (which is (\ref{HforS})) with a nondecreasing step function $\varphi$ in the locally average sense (which will be clear). Then approximate such $\varphi$ by strictly increasing smooth functions in the $L^1$-sense. Due to this $L^1$-approximation $\alpha$ can be chosen  in (\ref{bcat0}) at our disposal. 

\begin{proof}[Proof of Theorem \ref{Density}]
Choose any function from $\overline{\mathbb H}$. By de Branges \cite{deB2} and Winkler \cite{Win}, there is a unique matrix $H$ in $\V_+$ such that the corresponding $m$ function $m_H$ is the given Herglotz function.\\ 

Let us first approximate $H$ by $H_n$ such that they are projection matrices $P_{\varphi_n}$  whose $\varphi_n$ are nondecreasing step functions.  
Given $n\in\N$ put $I_{j,n}:=[\frac{j}{2^n},\frac{j+1}{2^n})$ and $H_{j,n}:=2^n\int_{I_{j,n}}H(x)\textrm{ }dx$, where $j=0,1,2,\cdots$. Then $H_{j,n}$ are constant, positive semidefinite $2\times2$ matrices with $\textrm{Tr }H_{j,n}=1$. Since $H_{j,n}$ are symmetric, by the spectral theorem, there are some real numbers $\varphi_{j,n}$ such that 
\begin{equation*}
H_{j,n}=\lambda_{j,n}P_{\varphi_{j,n}}+(1-\lambda_{j,n})P_{\varphi_{j,n}+\frac{\pi}2}
\end{equation*}
where $\lambda_{j,n}$ are eigenvalues of $H_{j,n}$ and $P_{\varphi_{j,n}}$  are orthogonal projections onto the eigenspaces for $\lambda_{j,n}$. The projections for the other eigenvalues are $P_{\varphi_{j,n}+\frac{\pi}2}$ because of the orthogonality of eigenspaces of two eigenvalues. If there is only one eigenvalue $\lambda_{j,n}$, it is possible to choose two orthogonal vectors in its eigenspace, since its multiplicity is two.\\

Construct $\varphi_n$ by 
\begin{equation*}
\varphi_n(x) := \begin{cases} \varphi_{j,n} & \quad x\in [\frac{j}{2^n}, \frac{j+\lambda_{j,n}}{2^n}) \\  \varphi_{j,n}+\frac{\pi}2 & \quad x\in [\frac{j+\lambda_{j,n}}{2^n},\frac{j+1}{2^n}) \end{cases}
\end{equation*}
in such a way that $\varphi_{j+1,n}\ge \varphi_{j,n}+\frac{\pi}2$ for all $j$. Indeed, if $\varphi_{j+1,n}< \varphi_{j,n}+\frac{\pi}2$ for some $j$, then add the smallest multiple of $\pi$ to $\varphi_{j+1,n}$ in order to make $\varphi_n$ nondecreasing.  Do this process from $j=0$ inductively, and denote new values by $\varphi_{j+1,n}$ again for convenience.  This is fine because we later deal with only three quantities  $\cos^2\varphi_n$, $\sin^2\varphi_n$ and $\cos\varphi_n\sin\varphi_n$ in (\ref{HforS}) which are periodic with the period $\pi$. Then $\varphi_n$ are nondecreasing step functions. See Figure 1 below.\\

\begin{tikzpicture}[scale=0.65]
\draw[->] (0,0) -- (8.5,0) node[anchor=north] { \large{$x$}};
\draw (0,0) node[anchor=north] {0}
          (1.9,0) node[anchor=north] {$\frac{\lambda_{0,n}}{2^n}$}
          (3.5,0) node[anchor=north] {$\frac1{2^n}$}
          (5.5,0) node[anchor=north] {$\frac{1+\lambda_{1,n}}{2^n}$}
          (7,0) node[anchor=north] {$\frac2{2^n}$}

          %(0,0.5)  node[anchor=east] {$\alpha$}
          (0,1.3) node[anchor=east] {$\varphi_{0,n}$}
          (0,2.3) node[anchor=east] {$\varphi_{0,n}$+$\frac{\pi}2$}
          (0,3.5) node[anchor=east] {$\varphi_{1,n}$}
          (0,4.5) node[anchor=east] {$\varphi_{1,n}$+$\frac{\pi}2$};
\draw[->] (0,0) -- (0,6) node[anchor=east] {\large{$\varphi_n$}}; 
\draw[dotted] (1.9,0) -- (1.9,6)
                        (3.5,0) -- (3.5,6)
                       (5.5,0) -- (5.5,6)
                       (7,0) -- (7,6);
\draw[very thick]  (0,1.3) -- (1.9,1.3)
          (1.9,2.3) -- (3.5,2.3)
          (3.5,3.5) -- (5.5,3.5)
          (5.5,4.5) -- (7,4.5)
          (7,5.7) -- (7.9,5.7);

\draw %(-0.03,0.5) -- (0.04,0.5)
          (-0.03,2.3) -- (0.04,2.3)
          (-0.03,3.5) -- (0.04,3.5)
          (-0.03,4.5) -- (0.04,4.5);
\draw (4.2,-.9) node[anchor=north] {\textbf{Figure 1. Step functions $\varphi_n$}};
\end{tikzpicture}

\noindent
Put $H_n:=P_{\varphi_n}$. 
The definition of $\varphi_n$ then indicates that, for given $n_0\in\N$,
\begin{equation}
\label{sameaverage}
\int_{I_{j,n_0}}H_{n}=\int_{I_{j,n_0}}H
\end{equation}
for all $n\ge n_0$. 
 
We next prove that $H_n$  converges weak-$\ast$ to $H$ in the sense of (\ref{weak*conv}). Let $f$ be a continuous (vector-valued) function with support contained in $[0,L]$ for some positive number $L$. By the Lebesgue lemma, for given $\epsilon>0$, there are numbers $M$, $n_0$ and $f_{j,n_0}$ such that, for all $x\in[0,L]$,
\begin{equation*}
  \sup \| f(x) \| \leq M 
\end{equation*}
 and, for all $n\ge n_0$,
\begin{equation*}
  \sup \| f(x)-f_{j,n_0} \| \leq \frac{\epsilon}{2ML} \quad \textrm{ for all } x,y\in I_{j,n}.  
\end{equation*}
Let us estimate the contribution of $H_n-H$  on each small interval. First decompose it by 
\begin{eqnarray*}
\int_{I_{j,n_0}}f^*(H_n-H) f &=& \int_{I_{j,n_0}}(f-f_{j,n_0})^* (H_n-H) f \\
 &+ & \int_{I_{j,n_0}}f_{j,n_0}^* (H_n-H) (f-f_{j,n_0})  \\
 &+&  \int_{I_{j,n_0}}f_{j,n_0}^*(H_n-H) f_{j,n_0}. 
\end{eqnarray*}
By (\ref{sameaverage}) the third integral is zero for all $n\ge n_0$. Now that the operator norm of $H_n-H$ is bounded by 2 (each $H$ is bounded by 1 due to $\textrm{Tr }H=1$ and positive semidefiniteness of $H$), the absolute values of the first and second integrals are bounded by $\frac{\epsilon}{L2^{n_0}}$. Hence  the quantity $\big| \int_0^\infty f^*(H_n-H)f \big|$ is less than $2\epsilon$. Since $\epsilon$ is arbitrary, $H_n$ converges to $H$ weak-$\ast$.\\

We have so far constructed nondecreasing step functions $\varphi_n$ such that $H_n$ ($=P_{\varphi_n}$)  converges to $H$ weak-$*$. These $\varphi_n$ are, however, not the ones corresponding to some Schr\"odinger equations, since $\varphi_n$ are not differentiable, not linear with slope 1 near $0$, and not strictly increasing. 

To overcome these, for each $n$ let us construct new functions $\tilde{\varphi}_{m,n}$ in the following way. For convenience the subscript $n$ is dropped, and so $\varphi_n$ and $\tilde{\varphi}_{m,n}$ are denoted by $\varphi$ and $\tilde{\varphi}_m$, respectively. Assume that all the steps of the graph of $\varphi$ are bounded. This is OK because if an unbounded step exists, then it would be the last step and it can be considered as a singular interval. This singular interval can then be converted to some boundary condition at the starting point of the unbounded interval, as talked in the proof of Proposition \ref{Cor}. 

Except for the first step, all other steps of $\varphi$ are approximated by piecewise linear, strictly increasing, and continuous functions $\tilde{\varphi}_m$, so that $\tilde{\varphi}_m$ converges to $\varphi$ in the $L^1$-sense. Look at the thick piecewise linear function in Figure 2 below. For the first (bounded) step, assume that  $\varphi(0) > \alpha$, where $\alpha$ is chosen at our disposal in (\ref{bcat0}). This is possible because of  shifting $\varphi$ up by $\pi$ at no costs. Then make $\tilde{\varphi}_m$ start at $(0, \alpha)$ and slightly move up linearly with slope 1. Do the similar procedure for its remaining part of the first step as the other steps. See the Figure 2 below. Then $\tilde{\varphi}_m$ are strictly increasing, piecewise linear, continuous functions with $\tilde{\varphi}_m(0)=\alpha$ which look linear with slope 1 near 0. The problem is, however, that they may not be differentiable.\\ 

\begin{tikzpicture}[scale=0.65]
\draw[->] (0,0) -- (8.5,0) node[anchor=north] { \large{$x$}};
\draw (0,0) node[anchor=north] {0}
          (1.9,0) node[anchor=north] {$\frac{\lambda_{0,n}}{2^n}$}
          (3.5,0) node[anchor=north] {$\frac1{2^n}$}
          (5.5,0) node[anchor=north] {$\frac{1+\lambda_{1,n}}{2^n}$}
          (7,0) node[anchor=north] {$\frac2{2^n}$}

          (0,0.5)  node[anchor=east] {$\alpha$}
          (0,1.3) node[anchor=east] {$\varphi_{0,n}$}
          (0,2.3) node[anchor=east] {$\varphi_{0,n}$+$\frac{\pi}2$}
          (0,3.5) node[anchor=east] {$\varphi_{1,n}$}
          (0,4.5) node[anchor=east] {$\varphi_{1,n}$+$\frac{\pi}2$};
\draw[->] (0,0) -- (0,6) node[anchor=east] {\large{$\tilde\varphi_m$}}; 
\draw[dotted] (0.15,0) -- (0.15,6)
                        (1.9,0) -- (1.9,6)
                        (3.5,0) -- (3.5,6)
                       (5.5,0) -- (5.5,6)
                       (7,0) -- (7,6);
\draw[thin]  (0,1.3) -- (1.9,1.3)
          (1.9,2.3) -- (3.5,2.3)
          (3.5,3.5) -- (5.5,3.5)
          (5.5,4.5) -- (7,4.5)
          (7,5.7) -- (7.9,5.7);
\draw (-0.03,0.5) -- (0.04,0.5)
          (-0.03,2.3) -- (0.04,2.3)
          (-0.03,3.5) -- (0.04,3.5)
          (-0.03,4.5) -- (0.04,4.5);

\draw[very thick] (0,0.5) -- (0.15,0.65) -- (0.25,1.2) -- (1.8,1.4) -- (2.0,2.2)    
                    -- (3.4,2.4) -- (3.6,3.4) -- (5.4,3.6) -- (5.6,4.4) -- (6.9,4.6) -- (7.1,5.6) -- (7.9,5.72);
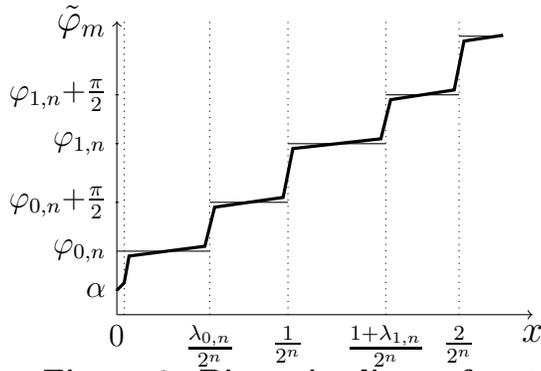
\draw (4.2,-0.9) node[anchor=north] {\textbf{Figure 2. Piecewise linear functions}};
\end{tikzpicture}

Mollifiers then enable us to make $\tilde{\varphi}_m$ smooth functions (and denote them by $\tilde\varphi_m$ again). Finally, the constructed functions $\tilde{\varphi}_m$ are smooth, strictly increasing functions which are linear with slope 1 near 0.  This means that $\tilde{\varphi}_m$ correspond to some Schr\"odinger equations by Proposition \ref{Cor}. See Figure 3 below.

\begin{tikzpicture}[scale=0.65]
\draw[->] (0,0) -- (8.5,0) node[anchor=north] { \large{$x$}};
\draw (0,0) node[anchor=north] {0}
          (1.9,0) node[anchor=north] {$\frac{\lambda_{0,n}}{2^n}$}
          (3.5,0) node[anchor=north] {$\frac1{2^n}$}
          (5.5,0) node[anchor=north] {$\frac{1+\lambda_{1,n}}{2^n}$}
          (7,0) node[anchor=north] {$\frac2{2^n}$}

          (0,0.5)  node[anchor=east] {$\alpha$}
          (0,1.3) node[anchor=east] {$\varphi_{0,n}$}
          (0,2.3) node[anchor=east] {$\varphi_{0,n}$+$\frac{\pi}2$}
          (0,3.5) node[anchor=east] {$\varphi_{1,n}$}
          (0,4.5) node[anchor=east] {$\varphi_{1,n}$+$\frac{\pi}2$};
\draw[->] (0,0) -- (0,6) node[anchor=east] {\large{$\tilde{\varphi}_m$}}; 
\draw[dotted] 
                        (1.9,0) -- (1.9,6)
                        (3.5,0) -- (3.5,6)
                       (5.5,0) -- (5.5,6)
                       (7,0) -- (7,6);
\draw[thin] (0,1.3) -- (1.9,1.3)
          (1.9,2.3) -- (3.5,2.3)
          (3.5,3.5) -- (5.5,3.5)
          (5.5,4.5) -- (7,4.5)
          (7,5.7) -- (7.9,5.7);
\draw (-0.03,0.5) -- (0.04,0.5)
          (-0.03,2.3) -- (0.04,2.3)
          (-0.03,3.5) -- (0.04,3.5)
          (-0.03,4.5) -- (0.04,4.5);

\draw[very thick, rounded corners]  (0,0.5) -- (0.15,0.65) -- (0.25,1.2) -- (1.8,1.4) -- (2.0,2.2)    
                    -- (3.4,2.4) -- (3.6,3.4) -- (5.4,3.6) -- (5.6,4.4) -- (6.9,4.6) -- (7.1,5.6) -- (7.9,5.72);
\draw (4.2,-0.9) node[anchor=north] {\textbf{Figure 3. Smooth functions}};
\end{tikzpicture}\\

Let us now rewrite the subscript $n$ which was dropped. In the above construction, observe that 
\begin{equation}
\label{L^1conv}
\| \tilde{\varphi}_{m,n}-\varphi_n \|_{L^1}\to 0
\end{equation}
as $m\to\infty$. 
Since sine and cosine are uniformly continuous, (\ref{L^1conv}) implies the weak-$*$ convergence of $\tilde{H}_{m,n}$ to $H_n$, where $\tilde{H}_{m,n}=P_{\tilde{\varphi}_{m,n}}$. Now that $H_n$ converges to $H$ weak-$*$, by Proposition \ref{CovofH},  $m_{\tilde{H}_{m,n}}$ converges to $m_H$ uniformly on compact subsets of $\C^+$ with suitably chosen $n$ and $m$. Theorem \ref{Density} has just been proven, since $m_{\tilde{H}_{m,n}}$ are Schr\"odinger $m$ functions which converge to the given Herglotz function, as desired.
\end{proof}

Remark. As talked in Sections \ref{secintro} and \ref{secDensity}, what has been seen is stronger than the main theorem; the proof above has showed that the Schr\"odinger $m$ functions with \textit{fixed $\alpha$} and \textit{smooth} potentials are dense in all Herglotz functions, since $\tilde{\varphi}_{m,n}(0)=\alpha$ at all times and $\tilde{\varphi}_{m,n}$  are all smooth, that is,  the corresponding potentials are smooth by Proposition \ref{Cor}.\\ 

\appendix
\section{Proof of Proposition \ref{CovofH}}
To be self-contained we are going to prove Proposition \ref{CovofH} by following the argument in \cite{Achrt}.  The extension of his argument from $\mathbb H$ to $\overline{\mathbb H}$, however, is necessary, especially to deal with $\infty$. As a cost of this extension, the proof becomes much longer than that in \cite{Achrt}.\\

Let's start with the following lemma. 
\begin{Lemma}
\label{convergenceforH}
Assume that $H_n$ converges to $H$ weak-$*$ in the sense of (\ref{weak*conv}), as $n\to\infty$, and let $u_n$ be solutions to (\ref{cs}) with $H_n$ such that the initial values $u_n(0)$ are the same for all $n$. Then the sequence $u_n$ has a subsequence which converges uniformly on any compact subsets of the half line $[0,\infty)$. Moreover, if $u$ is such a limit, then $u$ satisfies (\ref{cs}) with $H$. 
\end{Lemma}
\begin{proof}
Let us first see that $u_n$ are uniformly bounded on any bounded intervals in $[0,\infty)$, and then the sequence $u_n$ converges in a subsequence uniformly on any compact subsets. Given a positive number $R$,  assume that $|z|<R$ and put $\eta=\frac1{2R}$. Define operators $T_n$ from $C[0,\eta]$, the set of continuous (vector-valued) functions on $[0,\eta]$, to itself by
\begin{equation*}
(T_nu)(x)=-z\int_0^xJH_n(t)u(t)dt.
\end{equation*}
Then $||T_n||\leq 1/2$. Indeed, since $J$ is unitary, the trace-normed condition $\textrm{Tr }H=1$ and positive semidefiniteness imply that $\| JH_n \| \leq 1$ and  
\begin{equation}
\label{estT}
||T_n||
 = \textrm{ sup}_{||u||_{\infty}=1} \| T_nu(x) \| \leq  R\textrm{ }|x| \leq 1/2  
\end{equation}
for $x\in[0,\eta]$. 
In other words, $T_n$ are uniformly bounded in $n$. Now that the Neumann series converges and $u_n$ are solutions to (\ref{cs}) with $H_n$,  
\begin{equation}
\label{Neumann}
(I-T_n)^{-1}=\sum_{k=0}^{\infty}T_n^k,
\end{equation}
and 
\begin{equation}
\label{slnu_n}
u_n(x)-u_n(0)=(T_nu_n)(x), \textrm{ or} \quad u_n(x)=(I-T_n)^{-1}u_n(0). 
\end{equation}
Then (\ref{estT}) and (\ref{Neumann}) say that $|| (I-T_n)^{-1} ||\leq 2$, which implies that $u_n$ are uniformly bounded in $n$ on any bounded subsets of $[0,\infty)$. Due to (\ref{slnu_n}), $u_n$ are equicontinuous. By Arzela-Ascolli Theorem, the sequence $u_n$ has a subsequence which converges uniformly on compact subsets of $[0,\infty)$, say, $u_n\to u$, as $n\to\infty$. For convenience, we keep the same notation $u_n$ for the subsequence.

Let us see that $u$ satisfies (\ref{cs}) with $H$. Similar to (\ref{slnu_n}), observe that 
\begin{equation*}
u_n(x)-u_n(0)=-z\int_0^xJH_n(t)u_n(t)dt.
\end{equation*}
The left-hand side goes to $u(x)-u(0)$, as $n\to\infty$, by continuity. Due to splitting the integral on the right-hand side by  
\begin{eqnarray*}
&  & \int_0^xJH_n(t)u_n(t)dt \\
&=& \underbrace{ \int_0^xJH_n(t)\big(u_n(t)-u(t)\big)dt }_{=: I}
+\underbrace{\int_0^xJ\big( H_n(t)-H(t)\big)u(t)dt}_{=: II}\\
&  & +\int_0^xJH(t)u(t)dt
\end{eqnarray*}
it suffices to show that the first two integrals $I$ and $II$ go to zero as $n\to\infty$. First, $I\to 0$ because $JH_n(t)dt$ are finite measures on $[0,x]$, and the sequence $u_n$ converges to $u$ uniformly on $[0,x]$. To see that $II\to 0$, recognize our test function through  
\begin{equation*}
II=\int_0^{\infty}J\big( H_n(t)-H(t)\big)\chi_{[0,x]}(t)u(t)dt. 
\end{equation*}  
Since $H(t)dt$ is absolutely continuous with respect to the Lebesgue measure and, in particular, $H(t)dt$ does not have point masses, the weak-$*$ convergence works fine with any characteristic functions $\chi_{I}$ for any bounded interval $I$. Therefore $II$ goes to zero, as $n\to \infty$. This tells us that $u$ satisfies the equation 
\begin{equation*}
u(x)-u(0)=-z\int_0^xJH(t)u(t)dt, 
\end{equation*}
that is, $u$ is a solution to (\ref{cs}) with $H$, by the uniform convergence of $u_n$ on any compact subsets of $[0,\infty)$ and the fact that any solution to (\ref{cs}) is absolutely continuous \cite{KL,Win}.\\  
\end{proof}
We now prove Proposition \ref{CovofH}. 
\begin{proof}[Proof of Proposition \ref{CovofH}]
Choose a sequence $  H_n $ which converges to $H$ weak-$*$, as $n\to\infty$. Assume first that $H$ not be the same as the constant matrix $\big( \begin{smallmatrix} 1 & 0 \\ 0 & 0 \end{smallmatrix}\big)$ on $(0,\infty)$ and that $m_{H_n}$ not converge to $\infty$, even in a subsequence.  This case is discussed first because it gives us the basic idea how to prove Proposition \ref{CovofH}. It will be shown later that  $H_n$ converges weak-$*$ to  the constant matrix $\big( \begin{smallmatrix} 1 & 0 \\ 0 & 0 \end{smallmatrix}\big)$ on $(0,\infty)$ if and only if $m_{H_n}$ converges to $\infty$. Hence it would be enough to assume only that $H$ not be the constant matrix $\big( \begin{smallmatrix} 1 & 0 \\ 0 & 0 \end{smallmatrix}\big)$ here.

It is well known that Weyl theory can be applied to a canonical system in a similar way to that for Schr\"odinger operators (see e.g.  \cite{Achrt}).  More precisely, it is possible to choose $f_n=( f_{n,1} ,f_{n,2})^t$ which are $H_n$-integrable solutions to (\ref{cs})  (i.e., $f_n$ satisfy $\int_0^{\infty}f_n^*H_nf_n<\infty$) such that, for $z\in\C^+$, 
\begin{equation}
\label{L2slnform}
f_n(x,z)=u_n(x,z)+m_{H_n}(z)v_n(x,z)
\end{equation}
and  
\begin{equation}
\label{ImwithH}
\frac{\textrm{Im }m_{H_n}(z)}{\textrm{ Im }z}=\int_0^{\infty}f^*_n(x,z)H_n(x)f_n(x,z)dx  
\end{equation}
where $m_{H_n}$ are the $m$ functions of (\ref{cs}) with $H_n$, and $u_n$ and $v_n$ are the solutions for $H_n$ satisfying the initial conditions $u_n(0)=\big( \begin{smallmatrix} 1 \\ 0 \end{smallmatrix}\big)$ and $v_n(0)=\big( \begin{smallmatrix} 0 \\ 1 \end{smallmatrix}\big)$.
Due to (\ref{L2slnform}) and the initial values of $u_n$ and $v_n$, $m_{H_n}$ are expressed by 
\begin{equation}
\label{expform}
m_{H_n}(z)=\frac{f_{n,2}(0,z)}{f_{n,1}(0,z)}
\end{equation}
which is (\ref{mfnforcs}). 

The sequence $f_n$ then has a convergent subsequence. Indeed, the compactness of $\overline{\mathbb H}$ implies that $m_{H_n}$ has a convergent subsequence, say, $m_{H_n}(z)\to m(z)(\neq\infty)\in\overline{\mathbb{H}}$, uniformly on compact subsets of $\C^+$ (for convenience we use the same notation for a subsequence). Since  $H_n$ converges to $H$ weak-$*$,  Lemma \ref{convergenceforH} tells us that 
\begin{equation*}
u_n(x,z)\to u(x,z), \textrm{ and } v_n(x,z)\to v(x,z)
\end{equation*}
in subsequences uniformly on compact subsets in $x$ and $z$, and $u$ and $v$ are the solutions to (\ref{cs}) with $H$ satisfying  $u(0)=\big( \begin{smallmatrix} 1 \\ 0 \end{smallmatrix}\big)$ and $v(0)=\big( \begin{smallmatrix} 0 \\ 1 \end{smallmatrix}\big)$. 
By (\ref{L2slnform}), $f_n$ converges in a subsequence, say to $f$ which is written by  
\begin{equation}
\label{expform2}
f(x,z):=u(x,z)+m(z)v(x,z). 
\end{equation}

It is sufficient to show that $f$ is $H$-integrable. Because, if $f$ is $H$-integrable, the similar formula to (\ref{expform}) for $f$ and $m_H$ and (\ref{expform2}) indicate that  
\begin{equation*}
m(z)=\frac{f_{2}(0,z)}{f_{1}(0,z)}=m_{H}(z),
\end{equation*}
which says that $m_H$ is the only possible limit. This is followed by the fact that a trace-normed canonical system is always in a limit point case at $\infty$. Therefore  $m_{H_n}$ converges to $m_H$ uniformly on compact subsets of $\C^+$. 

Let's see $H$-integrability of $f$. Since $m(z)\neq\infty$, due to (\ref{ImwithH}), the sequence $\int_0^{\infty}f_n^*H_nf_n$ converges to some nonnegative number, as $n\to\infty$.  In particular,  the quantities $\int_0^{\infty}f_n^*H_nf_n$ are uniformly bounded for $n$, say, by $M$. Let us then see that, for all $n$,  
\begin{eqnarray*}
M &\ge& \int_0^{\infty}f_n^*H_nf_n \\
   &\ge& \int_0^Lf_n^*H_nf_n \quad \textrm{ for all positive number } L\\
   &=& \int_0^L(f_n-f)^*H_nf_n+\int_0^Lf_n^*H_n(f_n-f)+\int_0^Lf^*H_nf.
\end{eqnarray*}
Since $f_n$ converges to $f$ locally uniformly and $H(x)dx$ is absolutely continuous with respect to the Lebesgue measure, the weak-$*$ convergence of $H_n$ says that the first and second integrals go to zero and the third converges to $\int_0^Lf^*Hf$, as $n\to\infty$. By taking $L\to\infty$, $f$ is $H$-integrable, for these inequalities are all uniform in both $n$ and $L$.\\

 So far it has been showed that, if $H_n$ converges to $H$ weak-$\ast$, then $m_{H_n}$ converges to $m_H$, when the limit matrix $H$ is not the constant matrix $\big(\begin{smallmatrix} 1 & 0 \\ 0 & 0 \end{smallmatrix}\big)$ and $m_{H_n}$ does not converge to $\infty$, even in a subsequence.

It is the time to deal with the special case when the limit matrix $H$ is $\big( \begin{smallmatrix} 1 & 0 \\ 0 & 0 \end{smallmatrix}\big)$ on $(0,\infty)$, denoted by $H_{\infty}$, since its corresponding $m$ function is $\infty$. Let us show that $H_n$ converges to $H_{\infty}$ weak-$*$ if and only if $m_{H_n}$ converges to $\infty$. Before proving, note that, because of this equivalence, it was enough to assume that $H_n$ did converge to $H$ which was not $H_{\infty}$ in the previous case, as talked before. 

Observe that 
\begin{equation}
\label{zeroofv}
\int_0^L v^*Hv \textrm{ } dx = 0 \quad\textrm{if and only if}\quad H=H_{\infty} \textrm{ on } (0,L)
\end{equation}
where $v$ is the solution to (\ref{cs}) with $H$ satisfying $v(0)=\big( \begin{smallmatrix} 0 \\ 1 \end{smallmatrix}\big)$. 
Indeed, for any open interval $I$ in $(0,\infty)$ it is well known to have either that 
\begin{equation*}
\int_I e^*H e\textrm{ } dx=0,\textrm{ }  e\in\C^2 \textrm{ implies that } e=\big( \begin{smallmatrix} 0 \\ 0 \end{smallmatrix}\big)
\end{equation*}
(i.e., $I$ is \textit{of positive type} in \cite{HSW}),  
or that $I$ is a singular interval of type $\theta$, in other words, $H$ is the constant matrix $P_{\theta}$ 
\begin{equation}
\label{singularinterval}
P_{\theta}:=
\begin{pmatrix} \cos^2\theta& \cos\theta \sin \theta \\ \cos \theta \sin \theta & \sin^2\theta \end{pmatrix}
\end{equation}
almost everywhere on $I$ for some $\theta$ in $[0,\pi)$ (i.e., $I$ is \textit{$H$-invisible of type $\theta$} in \cite{HSW}). See lemma 3.1 in \cite{HSW} for more details.
Because of the continuity of $v$ and the initial condition $v(0)=\big( \begin{smallmatrix} 0 \\ 1 \end{smallmatrix}\big)$, if $\int_0^L v^*Hv \textrm{ } dx = 0$, then $\theta=0$ on $(0,L)$, which shows the sufficiency of (\ref{zeroofv}). Since $v_{\infty}(x)= \big( \begin{smallmatrix} 0 \\ 1 \end{smallmatrix}\big)$ on $(0,L)$, where $v_{\infty}$ is the solution to (\ref{cs}) with $H_{\infty}$ and the same initial condition of $v$, we have the necessity. 

Let us verify that,  if $H_n$ converges to $H_{\infty}$ weak-$*$, then $m_{H_n}$ converges to $\infty$. Suppose that $m_{H_n}$ does not converge to $\infty$. By the compactness we can choose $m\in\overline{\mathbb{H}}\setminus \{ \infty \} $ such that $m_{H_n}$ converges to $m$ at least in a subsequence. (For convenience we keep the same notation for the subsequence.)  Due to (\ref{L2slnform}) and (\ref{ImwithH}), see that, for any finite number $L>0$, 
\begin{eqnarray*}
\frac{\textrm{Im } m_{H_n}}{\textrm{ Im }z}
&=&\int_0^{\infty}f_n^*H_nf_n\\
&\ge& \int_0^{L}f_n^*H_nf_n\\
&=& \int_0^{L} u_n^* H_n u_n+m_{H_n}\int_0^{L} u_n^* H_n v_n\\
&& +\bar{m}_{H_n}\int_0^{L} v_n^* H_n u_n+|m_{H_n}|^2\int_0^{L} v_n^* H_n v_n.
\end{eqnarray*}
By taking $n\to\infty$, Lemma \ref{convergenceforH} tells that, at least in a subsequence, 
\begin{eqnarray*}
\frac{\textrm{Im } m}{\textrm{ Im }z}
&\ge&  \int_0^{L} u_{\infty}^* H_{\infty} u_{\infty}+m \int_0^{L} u_{\infty}^* H_{\infty} v_{\infty}\\
&& +\bar{m}\int_0^{L} v_{\infty}^* H_{\infty} u_{\infty}+|m|^2\int_0^{L} v_{\infty}^* H_{\infty} v_{\infty}\\
&=&\int_0^L 1 \textrm{ }dx
\end{eqnarray*}
for any finite number $L>0$, where $u_{\infty}$ and $v_{\infty}$ are the solutions for $H_{\infty}$ satisfying $u_{\infty}(0)=\big( \begin{smallmatrix} 1 \\ 0 \end{smallmatrix}\big)$ and $v_{\infty}(0)=\big( \begin{smallmatrix} 0 \\ 1 \end{smallmatrix}\big)$. Since the transfer matrix is $I_2-zxJP_0=\big(\begin{smallmatrix} 1&0\\-zx&1 \end{smallmatrix}\big)$ for $x>0$, direct computation says that $u_{\infty}(x)=\big( \begin{smallmatrix} 1 \\ -zx \end{smallmatrix}\big)$ and $v_{\infty}(x)=\big( \begin{smallmatrix} 0 \\ 1 \end{smallmatrix}\big)$, which implies the last equality. Due to choosing $L$ freely, the last integral, $\int_0^L 1\textrm{ } dx=L$, can be arbitrarily large, which contradicts that the left-hand side, $\frac{\textrm{Im } m}{\textrm{ Im }z}$, is a finite number which is independent of $L$. Therefore,  if $H_n$ converges to $H_{\infty}$ weak-$*$, then $m_{H_n}$ converges to $\infty$.

The remaining is to show that, if $m_{H_n}$ converges to $\infty$,  then $H_n$ converges to $H_{\infty}$ weak-$*$. Put $\tilde{f}_n=-\tilde{m}_{H_n}u_n+v_n$ where $\tilde{m}_{H_n}(z)=-\frac1{m_{H_n}(z)}$. In other words, since $m_{H_n}\to\infty$, (\ref{L2slnform}) cannot be used directly. Let us, instead, consider the negative reciprocals of $m_{H_n}$. Indeed, similar to (\ref{ImwithH}), it satisfies the equality 
\begin{equation}
\label{ImwithH2}
\frac{\textrm{Im }\tilde{m}_{H_n}(z)}{\textrm{ Im }z}=\int_0^{\infty}\tilde{f}^*_n(x,z)H_n(x)\tilde{f}_n(x,z)dx.
\end{equation}
For any finite number $L>0$, see that 
\begin{eqnarray*}
\frac{\textrm{Im }\tilde{m}_{H_n}}{\textrm{ Im }z} 
&=& \int_0^{\infty}\tilde{f}_n^* H_n\tilde{f}_n\\
&\ge& \int_0^{L}\tilde{f}_n^* H_n\tilde{f}_n\\
&=& |\tilde{m}_{H_n}|^2\int_0^{L} u_n^* H_n u_n-\overline{\tilde{m}}_{H_n}\int_0^{L} u_n^* H_n v_n\\
&& -\tilde{m}_{H_n}\int_0^{L} v_n^* H_n u_n+\int_0^{L} v_n^* H_n v_n.
\end{eqnarray*}
Since $\tilde{m}_{H_n} \to 0$, the left hand side on (\ref{ImwithH2}) converges to $0$. The compactness of $\V_+$ and Lemma \ref{convergenceforH} indicate that, for $L>0$, 
\begin{equation*}
\int_0^{L}\tilde{f}^*_n H_n\tilde{f}_n \rightarrow \int_0^L v^*Hv
\end{equation*} 
as $n\to \infty$, at least in a subsequence. Hence, whenever $H_n$ converges to $H$ weak-$*$ in a subsequence, we have shown that 
\begin{equation*}
\int_0^L v^*Hv=0
\end{equation*}
for any $L>0$, which, with (\ref{zeroofv}), implies that $H=H_{\infty}$ on $(0,\infty)$. Everything has then been done, since $H_{\infty}$ is the only possible limit.\\ 
\end{proof}

\textit{Acknowledgement.} I am grateful to Christian Remling for useful discussion in my  work and for all the help and support he has given me. It is also a pleasure to thank Matt McBride for important suggestions and friendly encouragement.  This research was partly supported by the National Research Foundation of Korea(NRF) grant funded by the Korea government(MOE) (No. 2014R1A1A2058848).\\

\end{document}